\documentclass[3p]{elsarticle}
\usepackage{bm,url}
\usepackage[usenames,dvipsnames,svgnames,table]{xcolor}
\usepackage{amssymb,amsmath,amsfonts}
\usepackage[lined,algonl,boxruled,titlenumbered]{algorithm2e}
\usepackage{multicol}

\usepackage{graphicx}
\graphicspath{{figures/}}
\usepackage[utf8x]{inputenc}

\usepackage[math]{easyeqn}

\graphicspath{{figures/}}

\newtheorem{theorem}{Theorem}[section]
\newtheorem{lemma}[theorem]{Lemma}

\newtheorem{remark}[theorem]{Remark}

\newtheorem{definition}[theorem]{Definition}

\newtheorem{problem}{Problem}
\newproof{proof}{Proof}

\newcommand{\pp}{\bm{p}}

\renewcommand{\tt}{\bm{t}}
\newcommand{\nn}{\bm{n}}
\newcommand{\vv}{\bm{v}}
\newcommand{\ww}{\bm{w}}
\newcommand{\sinc}{\mathop{\mathrm{sinc}}}
\newcommand{\cosc}{\mathop{\mathrm{cosc}}}
\newcommand{\vvv}{``}

\newcommand{\assign}{\leftarrow}
\newcommand{\arctandue}{\mathop{\mathrm{atan2}}}

\begin{document}
\begin{frontmatter}

\title{A Note on Robust Biarc Computation}

\author[DIMS]{Enrico Bertolazzi}
\ead[DIMS]{www.ing.unitn.it/~bertolaz}
\ead[DIMS]{enrico.bertolazzi@unitn.it}
\address[DIMS]{Department of Industrial Engineering -- University of Trento, Italy} 

\author[DISI]{Marco Frego}
\ead[DISI]{Marco Frego <marco.frego@unitn.it>}
\address[DISI]{Department of Information Engineering and Computer Science -- University of Trento, Italy}

\begin{abstract}
A new robust algorithm for the numerical computation of biarcs, 
i.e. $G^1$ curves composed of two arcs of circle, is presented.
Many algorithms exist but are based on geometric constructions, which 
must consider many geometrical configurations.
The proposed algorithm uses an algebraic construction which is 
reduced to the solution of a single $2$ by $2$ linear system.
Singular angles configurations are treated smoothly by 
using the pseudoinverse matrix when solving the linear system.
The proposed algorithm is compared with the Matlab's routine
\texttt{rscvn} that solves geometrically the same problem.
Numerical experiments show that Matlab's routine sometimes fails
near singular configurations and does not select the correct
solution for large angles, whereas the proposed algorithm
always returns the correct solution.
The proposed solution smoothly depends on the
geometrical parameters so that it can be easily included
in more complex algorithms like splines of biarcs or least squares data fitting.
\end{abstract}


\tnotetext[a]{\textbf{Acknowledgements}
This paper has received funding from the European Unions Horizon 2020 Research and Innovation 
Programme - Societal Challenge 1 (DG CON- NECT/H) under grant agreement $n^0$ 643644 \vvv ACANTO - A CyberphysicAl social NeTwOrk using robot friends''.
}
\end{frontmatter}

\section{Introduction}
In the industrial applications of curves there are two philosophies: one is the use of highly sophisticated polynomial splines or transcendental curves like high degree Bézier curves~\cite{Bezier:1970},
rational functions~\cite{Saini:2015,Zheng:2009,Farin:1999},
clothoid curves~\cite{Bertolazzi:2014,Narayan:2014,stoer:1982,Meek:1998,Walton:2009},
hodographs~\cite{Farouki:2015,Kozak:2015}, etc., which can produce continuous paths up to the curvature, the jerk or the snap (or even higher), but at a relatively expensive computational cost, usually because there are not closed form solutions and a system of nonlinear equations must be numerically solved. The other side of the coin is the employment of low degree polynomials, for instance piecewise linear interpolants or circular arc splines. The advantage of using this family of curves is that, at the price of losing some precision and smoothness, the computational times required to produce a path are in practice negligible, because the associated interpolation problem can be solved with elementary actions. Moreover, sometimes it is simply not necessary to go beyond $G^1$ continuity, a typical case is represented by real time applications.\\
In this paper we discuss an improvement of the algorithm for $G^1$ biarc fitting used in Matlab. A biarc is a curve obtained by connecting two arcs of circle that match with $G^1$ continuity and interpolate two given points and two angles.
Biarcs have several interesting properties, first of all, they are easy to understand and to use: in fact the arclength computation is straightforward, the tangent vector field is continuous and defined everywhere, the curvature is defined almost everywhere and is piecewise constant. Moreover, they are very useful in several applications, for instance, they are effectively used in the approximation of higher degree 
curves~\cite{Maier:2014,Deng:2014} or spirals \cite{Narayan:2014}, they easily  produce curves particularly used in CNC machining and milling, where the cutting devices follow the so called G-code, i.e. path composed of straight lines and circles.
Other applications of biarcs are in Computer Aided Design or Manufacturing (CAD-CAM), where they are used to specify the path \cite{Yang:2006} or the offset of a more general curve, \cite{Kim:2012}.
\\
\textbf{Related work.} Biarcs were originally proposed in an industrial environment rather than in an academic one, and from the 1970s they have been studied extensively by Bèzier \cite{Bezier:1970}, Bolton \citep{Bolton:1975} and Sabin \cite{Sabin:1977}. A general theoretical framework for a complete classification of the biarcs, in the M\"obius plane, is proposed in \cite{Kurnosenko:2013}. The solution of the biarc interpolation problem is not unique because the imposed constraints leave one degree of freedom, thus there is a one-dimensional family of interpolating biarcs to general planar $G^1$ Hermite data. Different choices of this free parameter give origin to different interpolation schemes. The most used construction techniques build the biarc by equal chord or by parallel tangent, \cite{Sir:2006}. In the first case the length of the two arcs is chosen equal, in the second case the tangent at the joint point is chosen parallel to the segment that connects the initial with the final point, \cite{Meek:1995,Meek:2008,Narayan:2014}. In all cases, it can be shown, see for instance \cite{Sir:2006}, all the possible joint points must be on a certain circle.
These solutions are based on a geometric approach and consider different cases (up to 128), as detailed in \cite{Kurnosenko:2013} and in the references therein contained. Typical cases are the C-shaped, S-shaped and J-shaped biarcs \cite{Kurnosenko:2013,Deng:2012}.\\
\textbf{Paper contribution.} The algorithm herein proposed extends the range of Hermite data where Matlab's implementation fails or gives a non-consistent solution. These cases are discussed with examples in Section~\ref{sec:num}.
Following our approach used for the solution of the $G^1$ Hermite Interpolation Problem with clothoid curves, \cite{Bertolazzi:2014}, we propose herein a novel pure analytic solution to the biarc problem, that does not require to split the problem in mutually exclusive cases. We select the free parameter required to close the system of equations in the same way of the Matlab's Curve Fitting Toolbox implementation (\cite{matlab:2017a}, page 12-218). The construction is explained in detail in the next section. The issue of Matlab's function for biarcs (\texttt{rscvn}) is that it cannot solve certain configurations of angles and that it gives a non-consistent solution for some range of angles. We show how to overcome this problem while maintaining  the same approach for the construction of the biarc. The solution is also extremely fast and numerically stable to be computed because only the solution of a $2$ by $2$ linear system is required. This is done via the explicit computation of the pseudoinverse \citep{Shinozaki:1972} matrix, which guarantees a consistent solution also in the case when the linear system is singular. The proposed algorithm is tested and validated in Section \ref{sec:num} and the complete pseudo code is given in~\ref{sec:algo}.

\section{Biarc Formulation}\label{sec:properties}
The biarc problem requires to find the pair of circle segments (possibly degenerate, as we will clarify next) that connect two points in the plane with assigned initial and final angles~\cite{Deng:2012}. More formally, it is the solution of the $G^1$ Hermite Interpolation Problem with two arcs. Let $\pp_0=(x_0,y_0)^T$ and $\pp_1=(x_1,y_1)^T$ be two points in the plane $\mathbb{R}^2$, $\vartheta_0$ and $\vartheta_1$ be the associated angles, then the biarc problem requires to find the solution of the following Boundary Value Problem (BVP):
\begin{EQ}[rclrclrcl]\label{eq:ODE}
  x'(\ell)      &=& \cos \theta(\ell), \qquad & x(0)&=&x_0, \qquad & x(L)&=&x_1,\\
  y'(\ell)      &=& \sin \theta(\ell), \qquad & y(0)&=&y_0, \qquad & y(L)&=&y_1, \\
  \theta'(\ell) &=& k(\ell),           \qquad & \theta(0)&=&\vartheta_0,  \qquad & \theta(L)&=&\vartheta_1,
\end{EQ}
where the curvilinear abscissa $\ell$ is in the range $[0,L]$.
The above equations ensure that the solution exhibits $G^1$ continuity, however, because there are not enough degrees of freedom, in general, it is not possible to satisfy \eqref{eq:ODE} with a single arc or straight line. Therefore, the curvature cannot be a continuous function and must be piecewise constant: 
\begin{EQ}\label{eq:ODE:k}
  k(\ell) = \begin{cases}
    \varkappa_0 & 0\leq \ell < \ell_\star  \\
    \varkappa_1 & \ell_\star \leq \ell \leq L
  \end{cases}
\end{EQ}
where we assume that the curvilinear abscissa $\ell$ runs from $0$ to $L$ and the curvature has a jump for 
$\ell_\star$, with $\ell_\star\in[0,L]$. The point for $\ell_\star$ is where the two arcs join. The two curvatures $\varkappa_0$, $\varkappa_1$ are real values, which can take the value zero. These values are associated to the radii of curvature of the two circles, if they are different from zero.
This formulation of the problem also contains degenerate cases, where the solution is not composed of two circles (i.e. we allow $\varkappa_0=0$ or $\varkappa_1=0$), meaning that a straight line can be part of the solution. Other particular cases are represented by a single arc of circle or by a single straight line.

As pointed out in several references, \cite{Meek:1995,Maier:2014, Kurnosenko:2013}, with this formulation the biarc solution is not unique, in fact the number of the constraints leaves one degree of freedom that allows many different geometric constructions \cite{Sir:2006}.

In this paper we focus on the solution proposed and implemented in Matlab's \texttt{rscvn} function, \cite{matlab:2017a}, page 12-218, which uses the degree of freedom to assign the direction of the (unit) normal vector $\nn(\ell)$ to the trajectory at $\ell_\star$:
\begin{EQ}\label{matlab:n}
  \nn(\ell_\star)=(-\sin\theta(\ell_\star),\cos\theta(\ell_\star))^T.
\end{EQ}
The consequence of assigning $\nn(\ell_\star)=\vv$ is that problem~\eqref{eq:ODE}--\eqref{eq:ODE:k}
will have at most one solution. According to Matlab's Handbook, such normal vector   
\begin{quote}
  \vvv $\vv$ is chosen as the reflection, across the perpendicular to the segment 
  from $\pp_0$ to $\pp_1$, of the average of the vectors $\nn(0)$ and $\nn(L)$''.
\end{quote}
We elaborate this construction by recasting it into an equivalent one expressed with
the tangent vectors $\tt(\ell)$.
The application of a rotation of $\pi/2$ to $\nn(\ell_\star)=\vv$ yields 
an equivalent condition $\tt(\ell_\star)=\ww$, where $\ww$ 
is reflected along the segment from $\pp_0$ to $\pp_1$, of the average of the tangents $\tt(0)$ and $\tt(L)$. 
Moreover, this construction can be improved by reasoning on the angles instead of the tangent vectors.
Indeed, it is more convenient to use the average of the angles rather than the average of the vectors, especially when the average of the vectors will yield a null (or very small) vector. In such cases the normal vector is not well posed, but the average of the angles is always well posed.\\
We construct $\ww$ on condition~\eqref{matlab:n} as $\ww=(\cos\vartheta_\star,\sin\vartheta_\star)^T$ and $\vartheta_\star$ is computed as, Figure \ref{fig:0}:
\begin{EQ}\label{eq:angle}\label{matlab:t}
  \overline{\vartheta_\star} = \dfrac{\vartheta_0+\vartheta_1}{2},
  \qquad
  \vartheta_\star = \alpha+(\alpha-\overline{\vartheta_\star}) = 2\alpha-\overline{\vartheta_\star}
\end{EQ}
with $\alpha = \mathop{\mathrm{atan2}}(y_1-y_0,x_1-x_0)$, e.g. $\alpha$ is the angle that satisfies
\begin{EQ}\label{eq:polar}
  \begin{cases}
  x_1-x_0=d\cos\alpha, \\
  y_1-y_0=d\sin\alpha,
  \end{cases}
  \quad 
  d=\norm{
  \begin{pmatrix}
    x_1-x_0\\
    y_1-y_0 
  \end{pmatrix}
  }.
\end{EQ}
The condition $\nn(\ell_\star)=\vv$ becomes thus 
$\theta(\ell_\star)=\vartheta_\star$.
\begin{figure}[!htb]
  \begin{center}
    \includegraphics[scale=0.7]{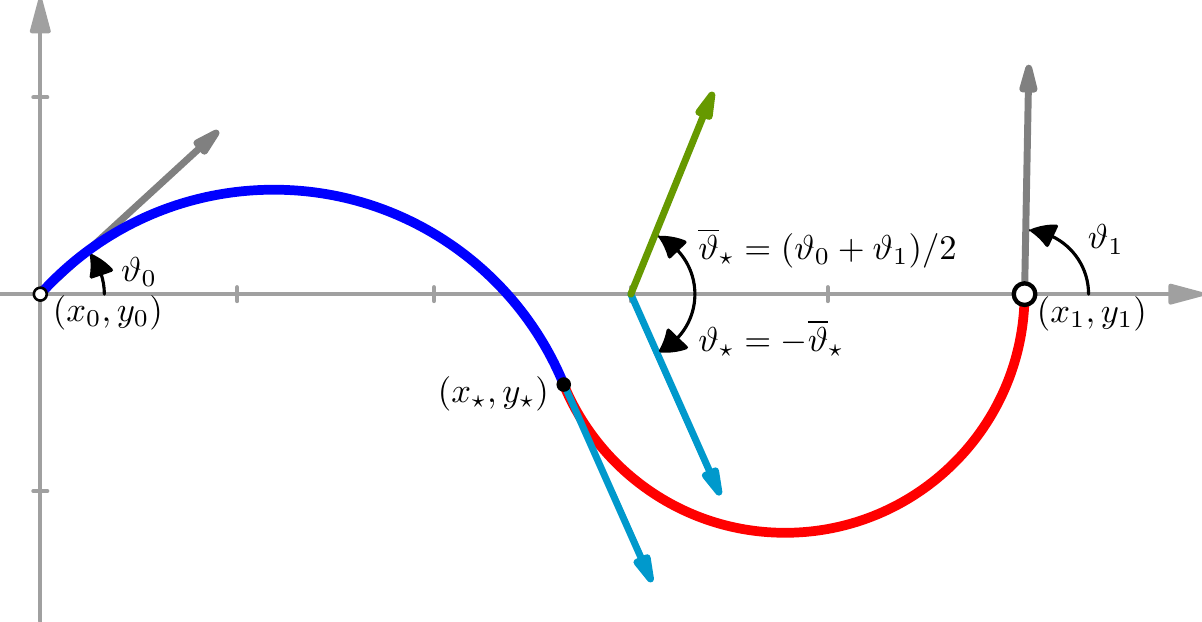}
  \end{center}
  \caption{Generalisation of Matlab biarc interpolation scheme, converted from normal vectors to tangent vectors.
  The figure shows the case of $\pp_0$ and $\pp_1$ aligned with the $x$ axis and $(x_\star,y_\star)$ the 
  joint point.}\label{fig:0}
\end{figure}

Now consider the Initial Value Problem (IVP) for the first segment of the biarc problem:
\begin{EQ}[rclrcl]\label{eq:IVP}
  x'(\ell)      &=& \cos \theta(\ell), \qquad & x(0)     &=&x_0,\\
  y'(\ell)      &=& \sin \theta(\ell), \qquad & y(0)     &=&y_0,\\
  \theta'(\ell) &=& \varkappa_0           \qquad & \theta(0)&=&\vartheta_0,
\end{EQ}
where $\varkappa_0\in\mathbb{R}$ is a constant value to be determined.

\begin{definition}\label{def:sc}
We define the functions $\sinc x$ and $ \cosc x$ as
\begin{EQ}\label{eq:smooth}
  \sinc x = \dfrac{\sin x}{x},\qquad
  \cosc x = \dfrac{1-\cos x}{x} 
\end{EQ}
that are used to find a numerically robust solution to~\eqref{eq:IVP}.
A standard way to compute~\eqref{eq:smooth} near the critical point $x=0$ is to
expand them with their Taylor approximations:
\begin{EQ}[rclrcl]
  \sinc x &=& 1-x\left(\dfrac{1}{6}-\dfrac{x^2}{20}\right)
  +\varepsilon_s(x),\quad &
  \abs{\varepsilon_s(x)}&\leq &\dfrac{\abs{x}^6}{5040}
  \\
  \cosc x &=& \dfrac{x}{2}
  \left(1-\dfrac{x^2}{12}\left(1-\dfrac{x^2}{30}\right)\right)
  +\varepsilon_c(x),\quad&
  \abs{\varepsilon_c(x)}&\leq &\dfrac{\abs{x}^7}{40320}.
\end{EQ}
Only a small number of terms must be considered for the required precision, for example, to limit the error for $\sinc x$ below $10^{-20}$ it is enough to have $\abs{x}\leq 0.002$, whereas for a (relative) error in the series of $\cosc x$ smaller than $10^{-20}$ it is enough to have $\abs{x}\leq 0.003$. They are implemented in Algorithms \ref{algo:Sinc} and \ref{algo:Cosc} in \ref{sec:algo}.
\end{definition}
By using definition~\ref{def:sc}, it is found by direct integration that the solution of~\eqref{eq:IVP} can be written as 
\begin{EQ}\label{eq:pre:nlsys}
  \begin{pmatrix}
    x(\ell) \\ y(\ell)
  \end{pmatrix}
  =
  \begin{pmatrix}
    x_0 \\ y_0
  \end{pmatrix}
  +
  \ell
  \begin{pmatrix}
    \cos\vartheta_0 & -\sin\vartheta_0 \\
    \sin\vartheta_0 & \cos\vartheta_0
  \end{pmatrix}
  \begin{pmatrix}
    \sinc(\varkappa\ell) \\
    \cosc(\varkappa\ell)
  \end{pmatrix}
  \qquad
  \theta(\ell) = \vartheta_0+\ell\varkappa,
\end{EQ}
where $\ell$ is the arc length of the curve.
In an analogous way we can compute the solution of the second arc that begins in $\pp_1$ with the corresponding angle and goes backwards from $\pp_1$ to meet the first segment. 
The biarc problem requires hence to find the point of intersection of the two curves and leads to the following problem definition.
\begin{problem}\label{prob:biarc}
The joint condition obtained with the Matlab condition~\eqref{matlab:t} (or equivalently~\eqref{matlab:n})
yields the nonlinear system:
\begin{EQ}\label{eq:pre:nlsys:1:nonstandard}
  \begin{cases}
    x(\ell_0;x_0,y_0,\vartheta_0,\varkappa_0)= x(-\ell_1;x_1,y_1,\vartheta_1,\varkappa_1),\\
    y(\ell_0;x_0,y_0,\vartheta_0,\varkappa_0)= y(-\ell_1;x_1,y_1,\vartheta_1,\varkappa_1),\\
    \vartheta_0+\ell_0\varkappa_0 = \vartheta_\star, \\
    \vartheta_1-\ell_1\varkappa_1 = \vartheta_\star,
  \end{cases}
\end{EQ}
where the unknowns are $\ell_0$, $\ell_1$, $\varkappa_0$ and $\varkappa_1$.
The function $x(\ell_0;x_0,y_0,\vartheta_0,\varkappa_0)$ is the solution of~\eqref{eq:pre:nlsys}
with initial values $x_0$, $y_0$, $\vartheta_0$, $\varkappa_0$ and analogously for the other functions. It is important to point out that $\ell_0>0$ and $\ell_1>0$.
\end{problem}
At this stage, it is convenient to recast the problem into standard form, by a transform that remaps the initial and the final points with the points $(0,0)$ and $(1,0)$, respectively. A similar bipolar transform is proposed also in \cite{Kurnosenko:2013, Bertolazzi:2014}. 

\begin{problem}[Standard form]\label{prob:biarc:standard}
The problem in standard form (after roto-translation and scaling) yields the nonlinear system:
\begin{EQ}\label{eq:pre:nlsys:1}
  \begin{cases}
  \big(\cos\theta_0\sinc(s\kappa_0)-\sin\theta_0\cosc(s\kappa_0)\big)s
  +\big(\cos\theta_1\sinc(-t\kappa_1)-\sin\theta_1\cosc(-t\kappa_1)\big)t=1 \\
  \big(\sin\theta_0\sinc(s\kappa_0)+\cos\theta_0\cosc(s\kappa_0)\big)s
  +\big(\sin\theta_1\sinc(-t\kappa_1)+\cos\theta_1\cosc(-t\kappa_1)\big)t=0\\
    \theta_0+s\kappa_0 = \theta_\star, \\
    \theta_1-t\kappa_1 = \theta_\star,
  \end{cases}
\end{EQ}
where using~\eqref{eq:polar} we obtain the following identity 
\begin{EQ}
  \theta_0     = \vartheta_0-\alpha,\quad
  \theta_1     = \vartheta_1-\alpha,\quad
  \theta_\star = \vartheta_\star-\alpha,\quad
  \kappa_0     = \varkappa_0 d,\quad
  \kappa_1     = \varkappa_1 d,\quad
  s = \ell_0/d,\quad
  t = \ell_1/d,
\end{EQ}
moreover the solution must satisfy $s>0$ and $t>0$. Notice that the standard assumption that the two points to be interpolated are different, i.e. $\pp_0\neq \pp_1$ implies $d>0$, hence $s$ and $t$ are well defined.
\end{problem}

\begin{lemma}\label{lem:regolare}
The solution $(s,t,\kappa_0,\kappa_1)$ of nonlinear system~\eqref{eq:pre:nlsys:1} in Problem~\ref{prob:biarc:standard} is obtained by solving the linear system
\begin{EQ}\label{eq:lsys}
  \bm{A}
  \begin{pmatrix} s \\ t \end{pmatrix}
  =
  \begin{pmatrix} 1 \\ 0 \end{pmatrix}
\end{EQ}
where $\bm{A}$ is a $2$ by $2$ matrix given by
\begin{EQ}\label{eq:sys:coeff}
  \begin{pmatrix}
    A_{11} \\ A_{21}
  \end{pmatrix}
  = 
  \begin{pmatrix}
    \cos\theta_0 & -\sin\theta_0 \\
    \sin\theta_0 & \cos\theta_0
  \end{pmatrix}
  \begin{pmatrix}
    \sinc\theta_\star^0 \\
    \cosc\theta_\star^0
  \end{pmatrix},
  \qquad
  \begin{pmatrix}
    A_{12} \\ A_{22}
  \end{pmatrix}
  =
  \begin{pmatrix}
    \cos\theta_1 & -\sin\theta_1\\
    \sin\theta_1 & \cos\theta_1
  \end{pmatrix}
  \begin{pmatrix}
    \sinc\theta_\star^1 \\
    \cosc\theta_\star^1
  \end{pmatrix},
\end{EQ}
and $\theta_\star^0 = \theta_\star - \theta_0$, 
$\theta_\star^1 = \theta_\star - \theta_1$. Finally
$\kappa_0=\theta_\star^0/s$ and $\kappa_1=-\theta_\star^1/t$.
\end{lemma}
\begin{proof}
From the last two equations of~\eqref{eq:pre:nlsys:1} we obtain
$s\kappa_0 = \theta_\star - \theta_0=\theta_\star^0$ and $-t\kappa_1 = \theta_\star-\theta_1=\theta_\star^1$.
The substitution of these relations into the first two equations of~\eqref{eq:pre:nlsys:1}
yields the linear system~\eqref{eq:lsys}.
\qed
\end{proof}
The solution of the linear system \eqref{eq:lsys} must be handled with care 
because of numerical instabilities that happen when the rank is not full and the determinant of the matrix of the coefficients is zero or close to zero. We discuss now these implications: first we consider the following determinants, used to 
theoretically solve the linear system by Cramer's Rule.

\begin{lemma}[Theoretical solution]
The solution $(s,t,\kappa_0,\kappa_1)$ of nonlinear system~\eqref{eq:pre:nlsys:1} of Problem~\ref{prob:biarc:standard}
is
\begin{EQ}\label{eq:nonlin:sol}
  s=d\,\dfrac{\mathcal{K}(\theta_0,\theta_\star)}
             {\mathcal{D}(\theta_\star^0,\theta_\star^1)},\qquad
  t=d\,\dfrac{\mathcal{K}(\theta_1,\theta_\star)}
             {\mathcal{D}(\theta_\star^0,\theta_\star^1)},
\end{EQ}
with $\kappa_0=\theta_\star^0/s$ and $\kappa_1=-\theta_\star^1/t$,
where $d$ is defined in~\eqref{eq:polar} and
\begin{EQ}
  \mathcal{D}(x,y) = \dfrac{\sin(x-y)+\sin y-\sin x}{xy},\qquad
  \mathcal{K}(x,y) = \dfrac{\cos x-\cos y}{x-y}.
\end{EQ}
\end{lemma}
\begin{proof}
We have the following determinants:
\begin{EQ}
  \abs{
    \begin{matrix}
      A_{11} & A_{12}  \\
      A_{21} & A_{22} 
    \end{matrix}
  }
  = \mathcal{D}(\vartheta_\star^0,\vartheta_\star^1),
  \quad
  \abs{
    \begin{matrix}
      1 & A_{12}  \\
      0 & A_{22} 
    \end{matrix}
  }
  =
  \mathcal{K}(\vartheta_0,\vartheta_\star),
  \quad
  \abs{
    \begin{matrix}
      A_{11} & 1 \\
      A_{21} & 0
    \end{matrix}
  }
  = 
  \mathcal{K}(\vartheta_1,\vartheta_\star),
\end{EQ}
and the thesis follows by employing Cramer's Rule for solving a linear system.\qed
\end{proof}
\begin{remark}
  The functions $\mathcal{D}(x,y)$ and $\mathcal{K}(x,y)$ can be evaluated via the identity
  \begin{EQ}
    \mathcal{D}(x,y) = \sinc y\cosc x-\sinc x\cosc y, \qquad
    \mathcal{K}(x,y) = -2\sin\left(\dfrac{x+y}{2}\right)\sinc\left(\dfrac{x-y}{2}\right)
  \end{EQ}
  Thus, the functions $\mathcal{D}(x,y)$ and $\mathcal{K}(x,y)$ can be computed 
  with the $\sinc$ and $\cosc$ expansions of Definition~\ref{def:sc} 
  and are well defined and numerically stable for all $x$ and $y$. 
\end{remark}
When the linear system \eqref{eq:lsys} has full rank and is far from singularity,
there are no numerical issues and the computation is safe.
It is important to notice, however, that the solution of the nonlinear system \eqref{eq:nonlin:sol} requires the ratio of those functions, which is not well defined when $\mathcal{D}(x,y)$ is close to zero. For instance, we have that
$\mathcal{D}(x,x)=0$ and thus the system associated to Problem \ref{prob:biarc} of biarc fitting has a singular configuration if $\vartheta_\star^0=\vartheta_\star^1$, that is if $\vartheta_1=\vartheta_0$. 
Another pathologic case is $\mathcal{K}(x,-x)=0$, which happens when the solution is degenerate, e.g. when the curvature becomes zero. This occurs when $\theta_i=-\theta_\star$, or expanding the previous term, 
if $\theta_i=(\theta_0+\theta_1)/2$, which implies again that $\theta_0=\theta_1$.
\begin{lemma}[Existence of the solution]\label{lem:singular}
Let $\theta_0$ and $\theta_1$ be angles in the interval $[-\pi,\pi]$ .
The solution $(s,t,\kappa_0,\kappa_1)$ of nonlinear system~\eqref{eq:pre:nlsys:1} in Problem~\ref{prob:biarc:standard}
exists if $\theta_0\neq\theta_1$. 
In the singular case $\theta_0=\theta_1=\theta$ the solution exists only if
the Matlab condition $\theta_\star=-\theta$ is satisfied and $\theta\in(-\pi,\pi)$.
\end{lemma}
\begin{proof}
In the singular case the coefficients of the linear system become
\begin{EQ}[rcl]
  \begin{pmatrix}
    A_{11} \\ A_{21}
  \end{pmatrix}
  =
  \begin{pmatrix}
    A_{12} \\ A_{22}
  \end{pmatrix}
  &=&
  \begin{pmatrix}
    \cos\theta & -\sin\theta \\
    \sin\theta & \cos\theta
  \end{pmatrix}
  \begin{pmatrix}
    \sinc(\theta_\star-\theta) \\
    \cosc(\theta_\star-\theta)
  \end{pmatrix}
  =
  \dfrac{1}{\theta_\star-\theta}
  \begin{pmatrix}
    \sin\theta_\star-\sin\theta \\
    \cos\theta-\cos\theta_\star
  \end{pmatrix}\\
  &=&\sinc(\theta_\star-\theta)
  \begin{pmatrix}
    \cos((\theta_\star+\theta)/2) \\
    \sin((\theta_\star+\theta)/2)
  \end{pmatrix}
\end{EQ}
and the system reduces to
\begin{EQ}
\sinc((\theta_\star-\theta)/2)
  \begin{pmatrix}
    \cos((\theta_\star+\theta)/2) \\
    \sin((\theta_\star+\theta)/2)
  \end{pmatrix}
  (s+t) = 
  \begin{pmatrix}
    1\\0
  \end{pmatrix}.
\end{EQ}
The only way to be consistent is that $\sin((\theta_\star+\theta)/2)=0$, i.e. 
$\theta_\star+\theta=0+2k\pi$ and due the limitation of range angle, $\theta_\star=-\theta$.
In this case we have that
\begin{EQ}
  \sinc\theta
  \begin{pmatrix}
    1 \\
    0
  \end{pmatrix}
  (s+t) = 
  \begin{pmatrix}
    1\\0
  \end{pmatrix},
\end{EQ}
which shows that the solution of the system exists and satisfies $s+t=1/\sinc\theta$
when $\sinc\theta\neq 0$.
Finally $\sinc\theta>0$ for $\theta\in(-\pi,\pi)$.\qed
\end{proof}
In conclusion, the solution of~\eqref{eq:pre:nlsys:1} exists by showing the solution of the linear system 
of Lemma~\ref{lem:regolare}, in the regular case.
In the singular case, when the Matlab condition is used, by lemma~\ref{lem:singular}
the linear system~\eqref{eq:lsys} is consistent and problem~\eqref{eq:pre:nlsys:1}
admits solutions.
Thus, even in the singular case it is possible to obtain a solution that makes 
sense in the geometric problem.
The least square solution of linear system~\eqref{eq:lsys} 
is chosen in the singular case.
The linear system is always solved with the
stable pseudoinverse computation which smoothly covers both the singular
and non-singular cases.
This computation is extremely fast due to the small dimension of the problem.

For a $2$ by $2$ matrix, the pseudoinverse can be computed directly, thus avoiding the need for
additional libraries and algorithms~\cite{Shinozaki:1972}.
Using LU factorisation of a $2$ by $2$ non zero matrix $\bm{A}=\bm{L}\bm{U}$
the pseudoinverse of $\bm{A}$ is easily computed by checking the only two cases (when $\bm{A}\neq\bm{0}$):
\begin{itemize}
   \item $\bm{L}$ and $\bm{U}$ are square and non-singular so that
         the pseudoinverse is equal to the usual inverse and    
         $\bm{A}^+=\bm{A}^{-1}= \bm{U}^{-1}\bm{L}^{-1}$.
   \item $\bm{L}$ and $\bm{U}$ are two vectors (row and column respectively).
         From the property $(\bm{L}\bm{U})^+ = \bm{U}^+\bm{L}^+$,
         $\bm{U}^+$ and $\bm{L}^+$ are computed using the formula
         $\bm{a}^+ = \bm{a}^T/\norm{\bm{a}}^2$ (valid when $\bm{a}$
         is a row or a column vector).
\end{itemize}
The complete biarc algorithm is implemented in Algorithm \ref{algo:biarc} in the Appendix, together with the pseudoinverse computation (Algorithm~\ref{algo:2x2}).
\section{Numerical Tests}
\label{sec:num}
In this section we show some numerical experiments to validate the presented algorithm. In the first test, see Figure \ref{fig:1}, we create a bouquet of biarcs all starting in $\pp_0=(0,0)$  with angles in the range $(-\pi,\pi)$ and ending at the point $\pp_1=(1,0)$ with different final angles.
\begin{figure}[!htb]
  \begin{center}
    \includegraphics[scale=0.7]{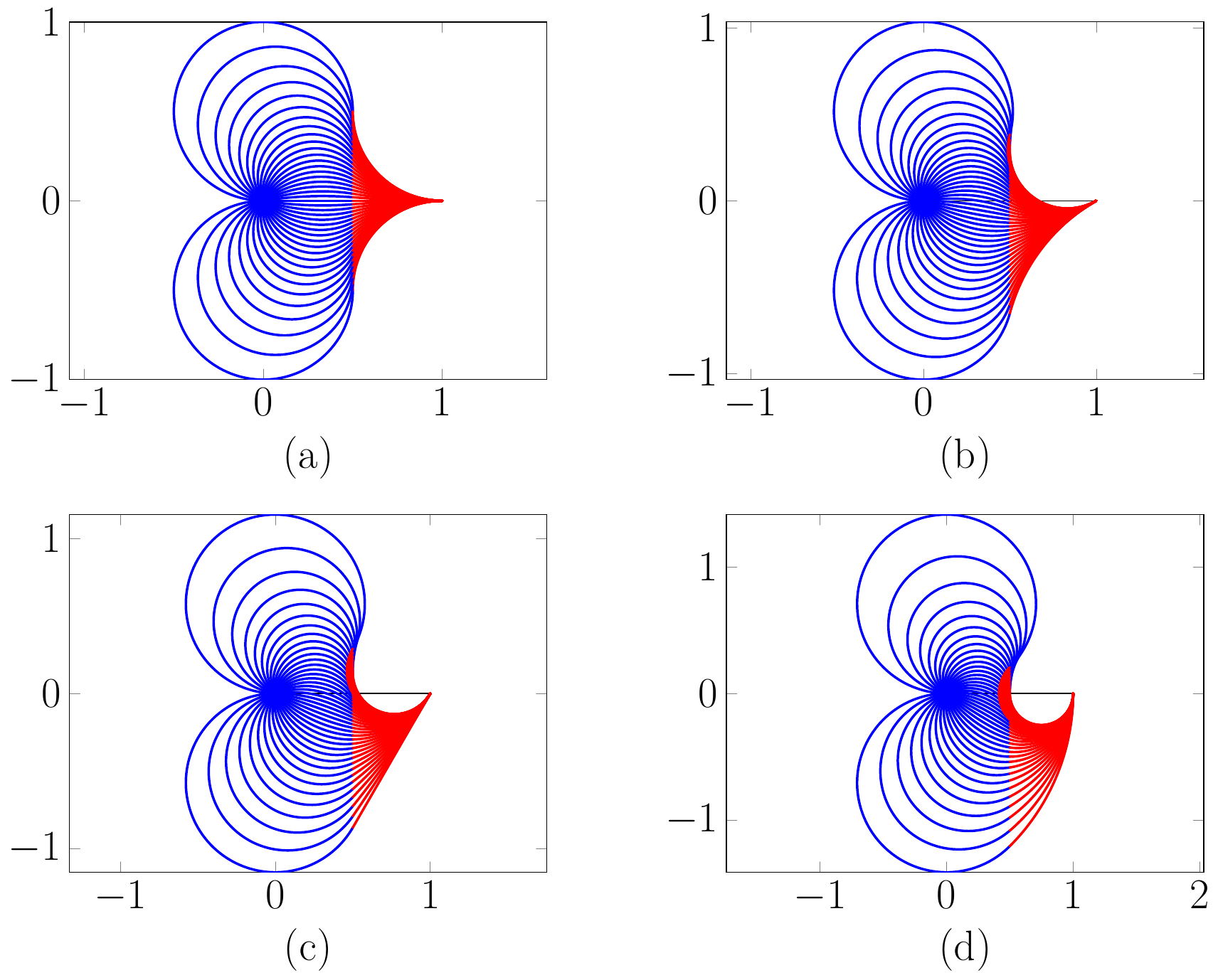}
  \end{center}
  \caption{Four examples of biarc interpolation with different initial and final angles. The first arc is plotted in blue, the second arc in red.}\label{fig:1}
\end{figure}
From Figure \ref{fig:1} we can see that the solution of the problem varies with continuity; in the following test we show that this is not the case with Matlab's function. In fact we can see in Figure \ref{fig:2} a direct comparison on the same tests between the algorithm herein proposed (cases (a) and (c)) and Matlab (cases (b) and (d)). In Figure \ref{fig:2} (a) and (c) there is continuity in the variation of the solution, whereas in 
Figure \ref{fig:2} (b) and (d) we can notice a jump in the solution, which is an undesirable behaviour.
\begin{figure}[!ht]
  \begin{center}
    \includegraphics[scale=0.7]{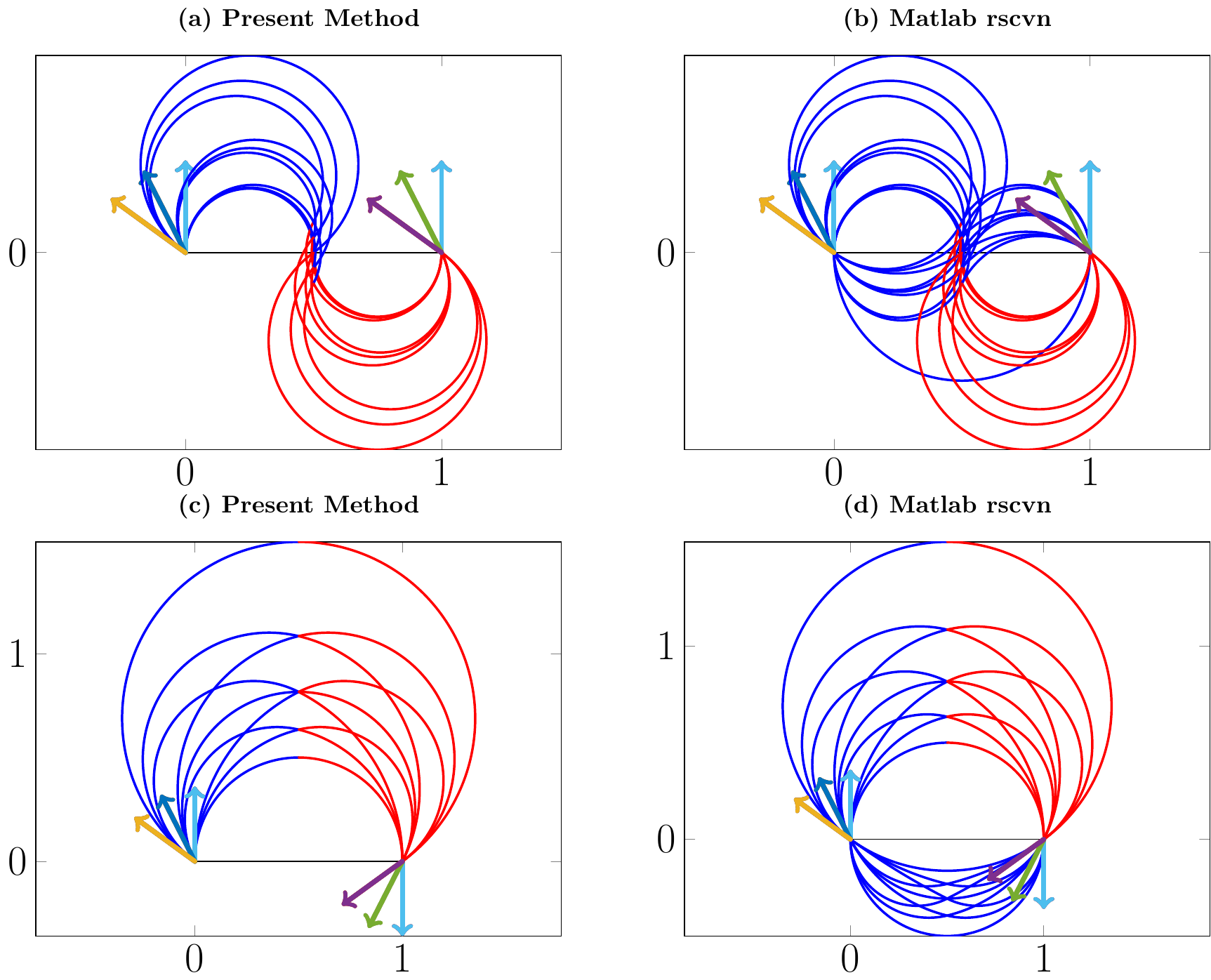}
  \end{center}
  \caption{Comparison between present method (a), (c) and Matlab (b) and (c). Arrows indicate the initial and final tangent vectors. Matlab's output exhibits wrong selections in the solution, which does not vary with continuity. }\label{fig:2}
\end{figure}
In Figure \ref{fig:2} (a) and (b) we plot the solutions for $\pp_0=(0,0)$ and $\pp_1=(1,0)$, the angles range in $[\pi/2, 4/5\pi]$, some tangent vectors are shown as arrows.
In Figure \ref{fig:2} (c) and (d) we plot the solutions for $\pp_0=(0,0)$ and $\pp_1=(1,0)$, the initial angles range in $[\pi/2, 4/5\pi]$, the final angles are in the range $[-4/5\pi,-\pi/2]$. In both cases (b) and (d) Matlab selects a non-natural solution.\\
As a last example, we show in Figure \ref{fig:3} two cases where Matlab produces a wrong solution when it is close to singular configurations, that is, when the average of the vectors used to find the joint point are zero or almost zero. In 
Figure \ref{fig:3} (a) our algorithm correctly interpolates $\pp_0=(0,0)$ and $\pp_1=(1,0)$ with $\vartheta_0=\vartheta_1=\pi/2$ producing a classic S-shaped biarc, while in (b), Matlab selects the wrong angle and produces a C-shaped biarc that violates the tangent at the initial point. In Figure \ref{fig:3} (c) and (d) we show the solution of the same problem with slightly perturbed angles:
$\pp_0=(0,0)$, $\pp_1=(1,0)$ but $\vartheta_0=\vartheta_1=\pi/2-10^4\epsilon$, where $\epsilon$ is the machine epsilon, i.e. a very small number. In Figure \ref{fig:3} (c) our algorithm produces a solution very close to the non-perturbed case (a), whereas Matlab gives a line segment, that is incompatible with the correct solution (c) or with the non-perturbed (still wrong) solution of (b).

\begin{figure}[!htb]
  \begin{center}
    \includegraphics[scale=0.7]{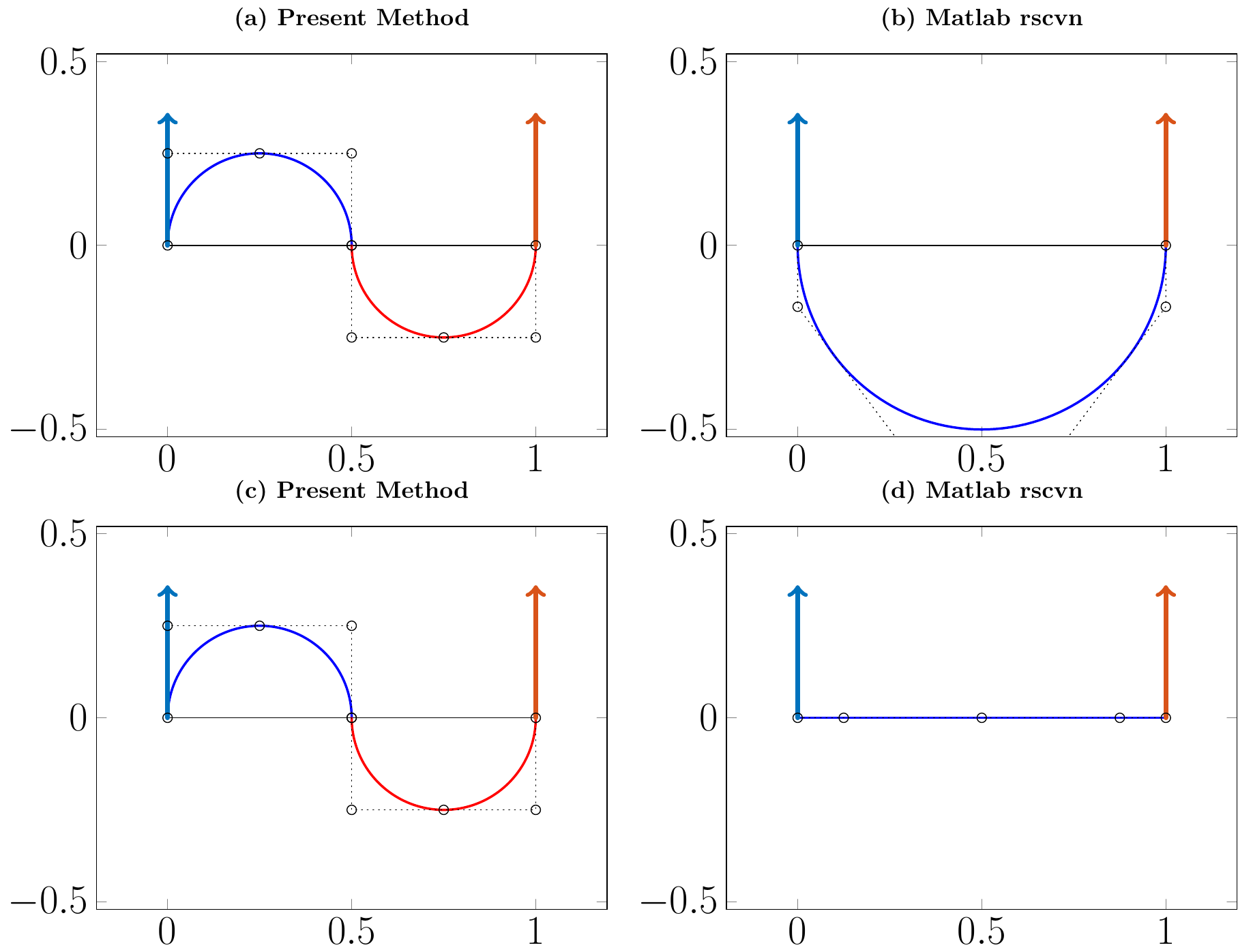}
  \end{center}
  \caption{Comparison between present method (a), (c) with Matlab (b) and (c). Arrows indicate the initial and final tangent vectors. Cases (a) and (b) are the non-perturbed angles $\vartheta_0=\vartheta_1=\pi/2$, cases (c) and (d) have $\vartheta_0=\vartheta_1=\pi/2-10^4\epsilon$, where $\epsilon$ is the machine epsilon. Dotted lines and dots are respectively the control polygons and control points.}\label{fig:3}
\end{figure}
\section{Conclusions}
\label{sec:conclusions}
A new robust algebraic algorithm for the numerical computation of biarcs is presented.
Differently to geometric based solution, it is not necessary
to consider many geometrical configurations.
The algorithm does not use any complex geometrical construction and is based on
the solution of a non linear system that is reduced to the solution
of a single $2$ by $2$ linear system.
The singular configuration (when the angles satisfy $\vartheta_0=\vartheta_1$)
is solved smoothly by using pseudoinverse when solving the linear system.
The Matlab's routine  \texttt{rscvn} solves geometrically the same problem; this has the drawback
that it is not possible to find the correct biarc in all the configurations.
Finally, \texttt{rscvn} fails to compute the biarc when the configuration
is \emph{almost singular}.
The biarc computed by the proposed algorithm smoothly depends on the
parameters e.g. $\vartheta_0$ and $\vartheta_1$ so that it can be easily included
in more complex algorithms like splines of biarcs or least squares data fitting.


\bibliographystyle{alpha}
\bibliography{bib_Biarc}

\appendix

\section{Complete Biarc Algorithm}
\label{sec:algo}

\setlength\columnsep{7pt} 

\begin{multicols}{2}
\small
\begin{algorithm}[H]
  \caption{Biarc solution}
  \label{algo:biarc}
  \SetKwFunction{Biarc}{Biarc}
  \SetKwFunction{Solve}{Solve2x2}
  \SetKwFunction{Cosc}{Cosc}
  \SetKwFunction{Sinc}{Sinc}
  \SetKwFunction{Range}{Range}
  \Biarc($x_0$, $y_0$, $\vartheta_0$, $x_1$, $y_1$, $\vartheta_1$)\;
  \Begin{
  \tcp{Transform to standard problem}
  $d_x\assign x_1-x_0;\:$
  $d_y\assign y_1-y_0$\;
  $d\assign\big(d_x^2+d_y^2\big)^{1/2};\;$
  $\alpha\assign\arctandue(d_y,d_x)$\;
  $\theta_0\assign\vartheta_0-\alpha;\;$
  $\theta_1\assign\vartheta_1-\alpha$\;
  $\theta_\star\assign-(\theta_1+\theta_0)/2$\;
  $\theta_\star^0 \assign \theta_\star-\theta_0;\;$
  $\theta_\star^1 \assign \theta_\star-\theta_1$\;
  $c_0\assign\cos\theta_0;\;$  $s_0\assign\sin\theta_0$\;
  $c_1\assign\cos\theta_1;\;$  $s_1\assign\sin\theta_1$\;
  \tcp{Compute joint point}
  $\begin{pmatrix} A_{11} \\ A_{21}\end{pmatrix} \assign \begin{pmatrix} c_0 & -s_0 \\ s_0 & c_0 \end{pmatrix}
  \begin{pmatrix} \Sinc(\theta_\star^0) \\  \Cosc(\theta_\star^0) \end{pmatrix}$\;
  $\begin{pmatrix} A_{12} \\ A_{22}\end{pmatrix} \assign \begin{pmatrix} c_1 & -s_1 \\ s_1 & c_1 \end{pmatrix}
  \begin{pmatrix} \Sinc(\theta_\star^1) \\  \Cosc(\theta_\star^1) \end{pmatrix}$\;
  $(s,t)\assign\Solve(\bm{A},(1,0)^T)$\;
  \tcp{Reverse transform}
  $\ell_0\assign d\,s;\;$
  $\ell_1\assign d\,t;\;$
  $\vartheta_\star\assign\theta_\star+\alpha$\;
  $\varkappa_0\assign\theta_\star^0/\ell_0;\;$
  $\varkappa_1\assign-\theta_\star^1/\ell_1$\;
  $c_a\assign\cos\alpha;\;$
  $s_a\assign\sin\alpha$\;
  $\begin{pmatrix} x_\star\\ y_\star\end{pmatrix} \assign
   \begin{pmatrix} x_0 \\ y_0\end{pmatrix} + \ell_0
   \begin{pmatrix} c_a & -s_a \\ s_a & c_a \end{pmatrix}
   \begin{pmatrix}  A_{11} \\ A_{21} \end{pmatrix}$\;
  \Return{$\ell_0$, $\varkappa_0$,
          $\ell_1$, $\varkappa_1$,
          $x_\star$, $y_\star$, $\vartheta_\star$}\;
  }
\end{algorithm}

\begin{algorithm}[H]
  \caption{$(\sin x)/x$ expansion}
  \label{algo:Sinc}
  \SetKwFunction{Sinc}{Sinc}
  \Sinc($x$)\;
  \Begin{
    \eIf{$\abs{x}<0.002$}
    {\Return $1+x\Big(\dfrac{1}{6}-\dfrac{x^2}{20}\Big)$}
    {\Return $(\sin x)/x$}
  }
\end{algorithm}

\begin{algorithm}[H]
  \caption{$(1-\cos x)/x$ expansion}
  \label{algo:Cosc}
  \SetKwFunction{Cosc}{Cosc}
  \Cosc($x$)\;
  \Begin{
    \eIf{$\abs{x}<0.002$}
    {\Return $\dfrac{x}{2}\left(1+\dfrac{x^2}{12}\left(1-\dfrac{x^2}{30}\right)\right)$}
    {\Return $(1-\cos x)/x$}
  }
\end{algorithm}
\SetAlgoCaptionLayout{tworuled}

\begin{algorithm}[H]
  \caption{Pseudoinverse $2$ by $2$}
  \label{algo:2x2}
  \SetKwFunction{Solve}{Solve2x2}
  \Solve($\bm{A}$, $\bm{b}$)\;
  \Begin{
    Let $k$ and $\ell$ be such that $\abs{A_{k\ell}}=\max\abs{A_{ij}}$\;
    Swap row $1$ with row $k$ \\ and column $1$ with column $\ell$ \\
    in the linear sistem $\bm{A}\bm{x}=\bm{b}$\;
    \If{$\abs{A_{11}}=0$}{\tcp{Null matrix, no solution}}
    $r\assign A_{21}/A_{11};\; w\assign A_{22}-rA_{12}$\;
    \eIf{$\abs{w}<\varepsilon$}
    {\tcp{find least squares solution}
     $t\assign (b_1 + rb_2) / ( (1+r^2)(A_{11}^2+A_{12}^2) )$\;
     $x_1\assign t A_{11};\;x_2\assign t A_{12}$\;
     \If{$\norm{\bm{A}\bm{x}-\bm{b}}>\varepsilon$}
     {\tcp{Inconsistent system}}
    }{
     $x_2\assign (b_2-r b_1)/w$\;
     $x_1\assign (b_1-A_{12}x_2)/A_{11}$\;    
    }
    Swap $x_1$ with $x_\ell$\;
    \Return $\bm{x}$\;
  }
\end{algorithm}

\end{multicols}

\end{document}